\documentclass[12pt,a4paper]{amsart}
\usepackage{a4wide}
\usepackage{amsfonts,amsthm,amsmath,amssymb,amscd}
\usepackage{graphicx,color}

\theoremstyle{plain}
\newtheorem{lem}{Lemma}[section]
\newtheorem{prop}[lem]{Proposition}
\newtheorem{thm}[lem]{Theorem}
\newtheorem*{thm*}{Theorem}
\newtheorem{cor}[lem]{Corollary}
\newtheorem*{cor*}{Corollary}

\newtheorem*{question*}{Question}
\theoremstyle{definition}
\newtheorem{defn}[lem]{Definition}
\newtheorem*{defn*}{Definition}
\newtheorem{ex}[lem]{Example}
\newtheorem*{ex*}{Example}
\newtheorem{rem}[lem]{Remark}
\newtheorem*{rem*}{Remark}

\newtheorem*{thmA}{Theorem A}
\newtheorem*{thmB}{Theorem B}
\newtheorem*{thmC}{Theorem C}

\theoremstyle{remark}

\DeclareMathOperator{\diam}{diam}
\DeclareMathOperator{\dist}{dist}

\DeclareMathOperator{\Leb}{Leb}

\DeclareMathOperator{\Sing}{Sing}

\newcommand{\C}{\mathbb C}
\newcommand{\D}{\mathbb D}
\newcommand{\clC}{\widehat \C}
\newcommand{\R}{\mathbb R}

\newcommand{\SSS}{\mathcal S}

\newcommand{\B}{\mathcal B}
\newcommand{\DD}{\mathcal D}

\newcommand{\PP}{\mathcal P}

\newcommand{\BB}{\mathcal B}

\begin{document}

\title{Conformal measures for meromorphic maps}

\date{\today}

\author{Krzysztof Bara\'nski}
\address{Institute of Mathematics, University of Warsaw,
ul.~Banacha~2, 02-097 Warszawa, Poland}
\email{baranski@mimuw.edu.pl}

\author{Bogus{\l}awa Karpi\'nska}
\address{Faculty of Mathematics and Information Science, Warsaw
University of Technology, ul.~Ko\-szy\-ko\-wa~75, 00-662 Warszawa, Poland}
\email{bkarpin@mini.pw.edu.pl}

\author{Anna Zdunik}
\address{Institute of Mathematics, University of Warsaw,
ul.~Banacha~2, 02-097 Warszawa, Poland}
\email{A.Zdunik@mimuw.edu.pl}

\thanks{Research supported by the Polish NCN grant 2014/13/B/ST1/04551, carried out at the University of Warsaw.}

\begin{abstract}
In this paper we study the relation between the existence of a conformal measure on the Julia set $J(f)$ of a transcendental meromorphic map $f$ and the existence of zero of the topological pressure function $t \mapsto P(f, t)$ for the map $f$. In particular, we show that if $f$ is hyperbolic and there exists a $t$-conformal measure which is not totally supported on the set of escaping points, then $P(f, t) = 0$. On the other hand, for a wide class of maps $f$, including arbitrary maps with at most finitely many poles and finite set of singular values and hyperbolic maps with at most finitely many poles and bounded set of singular values, if $P(f, t) = 0$, we construct a $t$-conformal measure on $J(f)$. This partially answers a question of R.~D.~Mauldin.
\end{abstract}

\maketitle

\section{Introduction}\label{introduction}

Let $f:\mathbb{C}\to \widehat{\mathbb{C}}$ be a transcendental meromorphic function. We denote by $f^n=f\circ \dots \circ f$ the $n$-th iterate of $f$. 
The \emph{Fatou set} $F(f)$ consists of all points $z\in\mathbb{C}$ for which there exists a neighbourhood $U$ of $z$ such that the family of iterates $\{f^n|_U\}_{n > 0}$ is defined and normal. The complement $\C\setminus F(f)$ is called the \emph{Julia set} and is denoted by $J(f)$. Note that some authors define $J(f)$ as  $\clC\setminus F(f)$, so that it always contains the point at infinity. In this paper we adopt the convention $\infty\notin J(f)$. Intuitively, the Julia set  carries the chaotic part of the dynamics of $f$. See e.g.~\cite{bergweiler} for a detailed presentation of the theory of iteration of transcendental meromorphic maps. 

In this paper we investigate ergodic properties of the dynamics of transcendental maps, using the tools of \emph{thermodynamic formalism}, developed by D.~Ruelle, R.~Bowen and P.~Walters in the 1970's and applied successfully to the study of the dynamics of rational maps on the Riemann sphere (see \cite{ruelle-book,pu-book} and the references therein). An important result in this area is the celebrated \emph{Bowen's formula}, which states that the Hausdorff dimension of the Julia set $J(f)$ of a hyperbolic rational map $f$ is equal to the unique zero 
of the \emph{topological pressure function} 
\begin{equation}\label{eq:bowen-form}
t \mapsto P(f, t) =
\lim_{n\to\infty} \frac{1}{n} \ln \sum_{w \in f^{-n}(z_0)}
|(f^n)'(w)|^{-t}
\end{equation}
for $z_0 \in J(f)$ (see \cite{rufusbowen}). Recall that a rational map $f$ is hyperbolic, if the closure in $\clC$ of the union of forward trajectories of all critical values of $f$ is disjoint from $J(f)$. 

The attempts to generalise the theory of thermodynamic formalism to the case of transcendental meromorphic maps started in the 1990's (see \cite{baranski-tangent,mauldin-urbanski}) and progressed in the subsequent years (see e.g.~\cite{KU1,urbanski-zdunik2,KU2,urbanski-zdunik1,KU3,KU4,skorulski,urbanski-zdunik3,mayer-urbanski-etds,BKZ-IMRN,bowen} and surveys \cite{ku,mayer-urbanski-mem}). 

In \cite{urbanski-zdunik2,urbanski-zdunik1}, M.~Urba\'nski and A.~Zdunik established the thermodynamic formalism theory for hyperbolic exponential maps of the form $z \mapsto \lambda e^z$, $\lambda \in \C \setminus \{0\}$. Recall that for the transcendental case, in the definition of a hyperbolic map $f$ one requires that the closure of the union of forward trajectories of all singular (critical and asymptotic) values of $f$ is bounded and disjoint from the Julia set (see Definition~\ref{defn:hyperb}). In the above two papers the authors discovered a crucial role of the \emph{radial Julia set} $J_r(f)$ (see Definition~\ref{def:radial}) in the study of the ergodic properties of transcendental maps. In particular, it turned out that the Bowen formula 
holds for a large class of transcendental maps in a modified form: the zero of the pressure function is equal to the Hausdorff dimension of the radial Julia set. As shown in \cite{rempe-radial}, the dimension of the radial Julia set coincides with the hyperbolic dimension of $J(f)$. Contrary to the case of rational dynamics, the radial Julia set of a transcendental map is often essentially smaller (in the sense of dimension) than the whole Julia set, even in the hyperbolic case. For instance, the Julia set of an arbitrary exponential map has Hausdorff dimension $2$ (see \cite{mcmullen-area}), while the Hausdorff dimension of the radial Julia set for a hyperbolic exponential map is greater than $1$ and smaller than $2$ (see \cite{urbanski-zdunik2,urbanski-zdunik1}). The fact that the Hausdorff dimension of $J_r(f)$ is 
greater than $1$ was generalised in \cite{BKZ-IMRN} to the case of transcendental meromorphic maps with logarithmic tracts over $\infty$ (in particular, for maps with a bounded set of singular values (class $\BB$), which are entire or meromorphic with a finite number of poles). See Definition~\ref{defn:tract} for a precise definition of a logarithmic tract. 

In \cite{baranski-tangent,KU1,KU2,KU3,skorulski}, elements of the thermodynamic formalism theory were established for other families of hyperbolic transcendental maps, both entire and meromorphic, including the sine and tangent family. The most general approach was presented in \cite{mayer-urbanski-etds,mayer-urbanski-mem}, where V.~Mayer and M.~Urba\'nski developed the thermodynamical formalism theory for hyperbolic transcendental meromorphic maps of finite order  
with the so-called \emph{balanced derivative growth condition}, which relates the growth of the derivative of the function (when its argument tends to infinity) with the growth of the function itself. This class includes many families of maps of the form $P(e^Q)$, where $P, Q$ are polynomials or rational functions.

Note that the pressure function defined in \eqref{eq:bowen-form} with the derivative of $f$ in the standard (Euclidean) metric is usually not suitable for transcendental case, since it can be infinite for all values of $t$. To overcome this difficulty, one considers derivative in some other conformal metric on $\C$. For instance, in \cite{mayer-urbanski-etds,mayer-urbanski-mem}, this metric has the form $d\rho = \frac{dz}{1 + |z|^\beta}$ for suitable $\beta \in \R$. 

In the previous paper \cite{bowen} the authors presented a general version of Bowen's formula in the context of transcendental dynamics, which holds for 
all transcendental meromorphic maps with a finite set of singular values (class $\SSS$) and a large class of hyperbolic and non-hyperbolic maps from class $\BB$.
The formula asserts that the Hausdorff dimension of the radial Julia set of a map $f$ is equal to a number $t_0$, which the infimum of the values $t > 0$ for which $P(f,t)$ is non-positive, where $P(f,t)$ is the pressure function defined in \eqref{eq:bowen-form}, with the derivative of $f$ taken in the spherical metric $ds = \frac{2dz}{1 + |z|^2}$. See Definition~\ref{defn:pressure} and Theorem~\ref{thm:pressure} for a precise formulation of the result. 
Note that similar results were obtained previously by F.~Przytycki, J.~Rivera-Letelier and S.~Smirnov \cite{conical, PRS} for arbitrary rational maps.

In the present paper, we consider another element of the thermodynamic formalism theory in the general setup of transcendental maps, investigating the question of the existence of a $t$-\emph{conformal measure}, i.e.~a Borel probability measure $\nu$ on the Julia set $J(f)$ such that
\[
\nu(f(A))=\int_A|f'(z)|^t d\nu(z)
\]
for every Borel set $A \subset \C$ on which $f$ is injective (see Definition~\ref{def:conformal}). This notion, introduced by S.~Patterson \cite{patterson} and D.~Sullivan \cite{sullivan}) in the context of Fuchsian and Kleinian groups, and developed in the papers by M.~Denker and M.~Urba\'nski (see e.g.~\cite{DU-parab,DU-conf,denker-urbanski}, proved to be extremely useful in many areas of conformal dynamics (see e.g.~\cite{LSS,DG,conical,BPS,graczyk-smirnov,VV,klaus} and surveys \cite{MU-measures,ku}).

In the context of the dynamics of rational maps, it is known that the Julia set of a rational map always admits a $t$-conformal measure for a suitable $t>0$, and the minimal exponent $t$ for which such a measure exists is equal to the Hausdorff dimension of the radial Julia set (see e.g.~\cite{conical}). 

In the case of transcendental maps the question of the existence of a $t$-conformal measure, where $t$ is the Hausdorff dimension of the radial Julia set, has been proved for several specific families of maps, e.g.~for the maps considered in \cite{urbanski-zdunik3,mayer-urbanski-etds}. However, the general question on the existence of such a conformal measure has been open.

Seeking an analogy between transcendental dynamics and the theory of infinite conformal iterated function systems (CIFS), developed by R.~D.~Mauldin and M.~Urba\'nski, recall that a conformal iterated function system is called \emph{regular}, if it admits a conformal measure on its limit set. In \cite[Theorem 3.5 and Lemma 3.13]{mauldin-urbanski} it was proved that a CIFS is regular if and only if there exists $t>0$ such that $P(t)=0$ where $P(t)$ is the topological pressure function defined for this system. In view of this, R.~D.~Mauldin (in a private communication) asked the following question. 

\begin{question*}
Let $f:\C \to \clC$ be a meromorphic transcendental function for which the pressure function $P(f,t)$ $($with the derivative in the spherical metric$)$ can be defined. Is the existence of a value $t>0$ such that $P(f,t)=0$ equivalent to the existence of a $t$-conformal measure on $J(f)?$
\end{question*}

Note that, in general, we do not know in which cases we have $P(f,t_0) = 0$ for $t_0$ defined above (although $P(f,t)$ is continuous and convex when it is finite, it could have a `jump' from the infinite value at $t = t_0$). It is known that $P(f,t_0) = 0$ for many functions $f$, including functions considered in \cite{mayer-urbanski-etds,mayer-urbanski-mem}, but it is an open question whether 
the opposite case can actually appear (cf.~Proposition~\ref{prop:clasif}).

In this paper we partially answer Mauldin's question. From now on, let
\[
P(f,t) = P(f,t,z_0)=\lim_{n\to\infty}\frac 1 n\ln\sum_{w\in
f^{-n}(z_0)}|(f^n)^*(w)|^{-t}
\]
for $t > 0$, where 
\[
f^*(z)  = \frac{(1+|z|^2)f'(z)}{1+|f(z)|^2}
\]
is the \emph{spherical derivative} of $f$, i.e.~the derivative with respect to  the \emph{spherical metric} defined by
\[
ds = \frac{2 \; dz}{1+|z|^2}. 
\]
Note that for the maps under consideration, the value of the pressure function $P(f,t,z_0)$ does not depend on $z_0 \in \C$ up to a set of Hausdorff dimension $0$ (see Theorem~\ref{thm:pressure}). The $t$-conformal measures $m_t$ we consider are also taken with respect to the spherical metric, i.e.~they satisfy
\[
m_t(f(A))=\int_A|f^*(z)|^t dm_t(z)
\]
for every Borel set $A \subset \C$ on which $f$ is injective.
We use the notion on the \emph{escaping set} of $f$ defined as
\[
I(f) = \{z \in \C: \{f^n(z)\}_{n=0}^\infty \text{ is defined and } f^n(z) \to \infty \text{ as } n \to \infty\}.
\]
The first result we prove is the following.

\begin{thmA}
If a hyperbolic meromorphic map $f: \C \to \clC$ admits a $t$-conformal measure $m_t$ for some $t > 0$, with respect to the spherical metric, then $P(f, t) \leq 0$. Moreover, if $m_t(J(f) \setminus I(f))>0$, then $P(f,t)=0$. 
\end{thmA}

The case $m_t(J(f) \setminus I(f))=0$ can actually occur, as shown in the following example.

\begin{ex} For $f(z) = \lambda \sin z$, $\lambda \in \C \setminus \{0\}$. Then 
$I(f)$ has positive $2$-dimensional Lebesgue measure (see \cite{mcmullen-area}) and the normalized $2$-dimensional spherical Lebesgue measure on $I(f)$ is $2$-conformal. Moreover, if additionally, $f$ is hyperbolic, then $P(f,2) < 0$ (see \cite{skorulski,mayer-urbanski-etds}).
\end{ex}

Denote by $\Sing(f)$ the set of all finite singular (critical and asymptotic) values of $f$ and let 
\[
\mathcal P(f) = \bigcup_{n=0}^\infty f^n(\Sing(f)),
\]
where we neglect terms which are not defined. 

Recall that we consider two classes of transcendental meromorphic functions:
\[
\mathcal S = \{f: \Sing(f) \text{ is finite}\}, \qquad
\mathcal B = \{f: \Sing(f) \text{ is bounded}\}.
\]

We will call a meromorphic map $f$ \emph{exceptional}, if there exists a (Picard) exceptional value $a$ of $f$, such that $a \in J(f)$ and $f$ has a non-logarithmic singularity over $a$. 

Theorem~A is an immediate corollary from the following more general result.

\begin{thmB} Let $f:\C \to \clC$ be a meromorphic map. Suppose that there exists  a $t$-conformal measure $m_t$ on $J(f)$, with respect to the spherical metric, for some $t > 0$. Then the following hold.
\begin{itemize}
\item[\rm (a)] If $f \in \mathcal S$, then:
\begin{itemize}
\item[$\circ$] $P(f, t) \le 0$ or there exists a  set $E \subset \overline{\PP(f)}$ of Hausdorff dimension $0$, such that $m_t(E)=1$,
\item[$\circ$] $P(f, t) \ge 0$ or $f^n(z) \to \PP(f) \cup \{\infty\}$ as $n \to \infty$ for $m_t$-every $z$.
\end{itemize}
\item[\rm (b)] If $f \in \mathcal B$ is non-exceptional and $J(f) \setminus\overline{\PP(f)} \neq \emptyset$, then:
\begin{itemize}
\item[$\circ$] $P(f, t) \leq 0$, 
\item[$\circ$] $P(f, t) \ge 0$ or $f^n(z) \to \overline{\PP(f)} \cup \{\infty\}$ as $n \to \infty$ for $m_t$-every $z$.
\end{itemize}
\end{itemize}
\end{thmB}

The next result we prove in this paper is the following (see Definition~\ref{defn:tract} for the definition of the logarithmic tract).

\begin{thmC} Let $f:\C \to \clC$ be a meromorphic map from class $\SSS$ with a logarithmic tract over $\infty$ or a non-exceptional map from class $\B$ with a logarithmic tract over $\infty$, such that $J(f) \setminus\overline{\PP(f)} \neq \emptyset$ (e.g. a hyperbolic map $f$, which is entire or has a finite number of poles). If $P(f, t) = 0$ for some $t > 0$, then there exists a $t$-conformal measure $m_t$ with respect to the spherical metric. Moreover, 
\[
m_t(\C \setminus \D(r))=  o\left(\frac{(\ln r)^{3t}}{r^t}\right) \quad \text{as} \quad r \to \infty,
\]
where $\D(r) = \{z \in \C: |z| < r\}$.
\end{thmC}

\begin{rem}\label{rem:thmB} In fact, the proof gives
\[
\int_{0}^\infty \frac{r^t}{(\ln r)^{3t}} \: m_t(\C \setminus \D(r)) < \infty.
\]
\end{rem}

\begin{rem}
Theorems A and B contribute to answer the R.~D.~Mauldin question in both directions. A main problem which remains to be determined, is under which condition the measure $m_t$ constructed in Theorem~C satisfies $m_t(J(f) \setminus I(f))=0$.
\end{rem}

\section{Preliminaries}\label{sec:prelim}

\subsection*{Notation}

In all definitions and formulations of results we assume that $f$ is a transcendental meromorphic function on the complex plane.

By a \emph{conformal metric} we mean a Riemannian metric on $\C$ of the form
\[
d\rho=\rho dz,
\]
where $dz$ is the standard (Euclidean) metric and $\rho$ is a continuous positive function on $\C$. The derivative of a map $f$ with respect to the metric $d\rho$ is equal to 
\begin{equation}\label{eq:der}
f'_\rho(z) = \frac{\rho(f(z))}{\rho(z)} f'(z),
\end{equation}
where $f'$ is the standard derivative. 
In particular, we consider the \emph{spherical metric} defined by
\[
ds = \frac{2 \; dz}{1+|z|^2} 
\]
and the \emph{spherical derivative}
\[
f^*(z): = f'_s(z) = \frac{(1+|z|^2)f'(z)}{1+|f(z)|^2}.
\]
The \emph{spherical distance} in $\clC$ (defined by the spherical metric) will be denoted by $\dist_{sph}$. 

By $\D(z,r)$ (resp.~$\DD(z,r)$) we denote the disc centred at $z \in \C$ (resp.~$z \in \clC$) of radius $r > 0$ with respect to the Euclidean (resp.~spherical) metric. For simplicity, we write $\D(r)$ for $\D(0, r)$.

\begin{defn}\label{defn:hyperb}
We write $\Sing(f)$ for the \emph{singular set} of $f$, which consists of all finite singular (critical and asymptotic) values of $f$ and define the \emph{post-singular set} by
\[
\mathcal P(f) = \bigcup_{n=0}^\infty f^n(\Sing(f)),
\]
where we neglect terms which are not defined. 

We say that $f$ is \emph{hyperbolic}, if $\overline{\mathcal P(f)}$ is bounded and disjoint from the Julia set of $f$.
We set
\[
\mathcal S = \{f: \Sing(f) \text{ is finite}\}, \qquad
\mathcal B = \{f: \Sing(f) \text{ is bounded}\}.
\]
Note that hyperbolic maps are in the class $\mathcal B$.
\end{defn}

\begin{defn}
We call $f$ \emph{exceptional}, if there exists a (Picard) exceptional value $a$ of $f$, such that $a \in J(f)$ and $f$ has a non-logarithmic singularity over $a$. 
\end{defn}

\subsection*{Radial Julia sets and conformal measures}

\begin{defn}\label{def:radial}
The \emph{radial Julia set} $J_r(f)$ is the set of points $z\in J(f)$ for which all iterates $f^n(z)$ are defined and there exist
$r=r(z)>0$ and a sequence $n_k\to\infty$, such that a holomorphic 
branch of $f^{-n_k}$ sending $f^{n_k}(z)$ to $z$ is well-defined on $\DD(f^{n_k}(z), r)$. 

We denote by $I(f)$ the \emph{escaping set} of $f$, i.e.
\[
I(f) = \{z \in \C: \{f^n(z)\}_{n=0}^\infty \text{ is defined and } f^n(z) \to \infty \text{ as } n \to \infty\}.
\]
\end{defn}

We consider conformal measures with respect to some conformal metrics on $\C$.

\begin{defn}\label{def:conformal} Let $f: \C \to \clC$ be a meromorphic map. We say that a Borel probability measure $\nu$ on the Julia set $J(f)$ is $t$-\emph{conformal} for some $t > 0$ with respect to a conformal metric $d\rho=\rho dz$, if 
\begin{equation}\label{conformality}
\nu(f(A))=\int_A|f_\rho'(z)|^t d\nu(z)
\end{equation}
for every Borel set $A \subset \C$ on which the map $f$ is injective.
\end{defn}

The notion of $t$-conformality is in a sense independent of the chosen metric, as shown in the following proposition. 

\begin{prop}\label{prop:conf_metric} Let $\nu$ be a $t$-conformal measure with respect to a conformal metric $d\rho_1$ and let $d\rho_2=\rho_2 dz$ be another conformal metric, such that 
$
M = \int \eta \, d\nu < \infty
$
for $\eta(z)=\left(\frac{\rho_2(z)}{\rho_1(z)}\right)^t$. Then the measure $\mu$ defined as 
\[
d\mu=\frac \eta M  \, d\nu
\]
is $t$-conformal with respect to the metric $d\rho_2$.
\end{prop}
\begin{proof} By \eqref{eq:der}, 
\begin{equation}\label{eq:rho}
f'_{\rho_1}(z) = \frac{\rho_1(f(z))\rho_2(z)}{\rho_2(f(z))\rho_1(z)} f'_{\rho_2}(z).
\end{equation}
Take a Borel set $A \subset C$ on which the map $f$ is injective.
Using the $t$-conformality of $\nu$ and \eqref{eq:rho}, we obtain, 
\begin{align*}
\mu(f(A))&= \frac 1 M \int_{f(A)}\eta(z)d\nu(z)
=\frac 1 M\int_A\eta(f(z)) |f'_{\rho_1}(z)|^t d\nu(z)\\
&=\frac 1 M\int_A \left(\frac{\rho_2(f(z))}{\rho_1(f(z))}\right)^t |f'_{\rho_1}(z)|^td\nu(z)=\frac 1 M\int_A \left(\frac{\rho_2(z)}{\rho_1(z)}\right)^t |f'_{\rho_2}(z)|^t d\nu(z)\\
&=\frac 1 M\int_A \eta(z) |f'_{\rho_2}(z)|^t d\nu(z)=\int_A|f'_{\rho_2}(z)|^t d\mu(z).
\end{align*}
\end{proof}

\begin{rem}
Note that in Proposition~\ref{prop:conf_metric}, if $\int \eta \, d\nu = \infty$, then the measure $d\mu=\eta  \, d\nu$ satisfies \eqref{conformality} but is infinite.
\end{rem}

As noted in the introduction, in this paper we consider $t$-conformal measures taken with respect to the spherical metric.

\subsection*{Distortion estimates}

We use the following spherical version of the classical Koebe Distortion Theorem, see, e.g.~\cite{bowen} for its detailed proof.
\begin{thm}[Spherical Distortion Koebe Theorem]\label{koebe}
Let $0 < r_1, r_2 < \diam_{sph} \clC$. Then there exists a constant $c>0$ depending only on $r_1, r_2$, such that for every spherical disc $D = \DD(z_0, r)$ and every univalent holomorphic map $g: D \to \clC$ with $z_0 \in \clC$, $\diam_{sph} D < r_1$ and $\diam_{sph}(\clC\setminus g(D))>r_2$, if $z_1, z_2 \in \DD(z_0, \lambda r)$ for some $0 < \lambda < 1$, then 
\[
\frac{|g^*(z_1)|}{|g^*(z_2)|} \leq\frac{c}{(1 - \lambda)^4}.
\]
\end{thm}

We recall the notion of a logarithmic tract, and formulate some distortion estimates which will be used in subsequent sections.
\begin{defn}\label{defn:tract}
Let $U \subset \C$ is an unbounded simply connected domain, such that the boundary of $U$ in $\C$ is a smooth open simple arc and let $R > 0$. If $f:\overline{U} \to \C$ is a continuous map, holomorphic on $U$, such that $|f(z)| = R$ for every $z$ in the boundary of $U$ and $f$ on $U$ is a universal covering of $\{z\in \C: |z| > R\}$, then we call $U$ a \emph{logarithmic tract} of $f$ over $\infty$.
\end{defn}

\begin{rem}\label{rem:log}
If  a map $f \in \B$ is entire or if it has a finite number of poles, then every component of $f^{-1}(V)$, where $V = \{z\in\C: |z| > R\}$ for sufficiently large $R$, is a logarithmic tract of $f$ over $\infty$.
\end{rem}

We shall make use of the following facts (see \cite{stallard, bowen} for the proofs). 

\begin{lem}[{\cite[Corollary 3.7]{bowen}}] \label{lem:|g|}
Let $R, L > 1$. Then there exists a constant $c > 0$ depending only on $R, L$, such that for every logarithmic tract $U \subset \C$ of $f:U \to V$ over $\infty$, where $V = \{z\in\C:|z| > R\}$ and $0 \notin U$, for every $z_1, z_2 \in V$ with $|z_1| \geq |z_2| \geq LR$ and every branch $g$ of $f^{-1}$ in a neighbourhood of $z_1$ $($or $z_2)$, we have
\[
c^{-1}\left(\frac{\ln|z_1|}{\ln|z_2|}\right)^{-4\pi} < \frac{|g(z_1)|}{|g(z_2)|} < c\left(\frac{\ln|z_1|}{\ln|z_2|}\right)^{4\pi} 
\]
for some extension of the branch $g$ to a neighbourhood of $z_2$ $($or $z_1)$.
\end{lem}

\begin{lem}[{\cite[Corollary 3.9]{bowen}}]\label{lem:|g^*|} Let $R, L > 1$. Then there exists a constant $c > 0$ depending only on $R, L$, such that for every logarithmic tract $U\subset \C$ of $f:U \to V$ over $\infty$, where $V = \{z\in\C: |z| > R\}$ and $0 \notin U$, for every $z_1, z_2 \in V$ with $|z_1| \geq |z_2| \geq LR$ and every branch $g$ of $f^{-1}$ in a neighbourhood of one of the points $z_1, z_2$, we have
\[
c^{-1} \frac{|z_1|}{|z_2|} \left(\frac{\ln |z_1|}{\ln |z_2|}\right)^{-3} \leq \frac{|g^*(z_1)|}{|g^*(z_2)|} \leq c \frac{|z_1|}{|z_2|} \frac{\ln |z_1|}{\ln |z_2|},
\]
for some extension of the branch $g$.
\end{lem}

\subsection*{Pressure for transcendental maps} 

\begin{defn}\label{defn:pressure} The \emph{topological pressure} function with respect to the spherical metric is defined as
\[
P(f,t,z_0)=\lim_{n\to\infty}\frac 1 n\ln\sum_{w\in
f^{-n}(z_0)}|(f^n)^*(w)|^{-t}
\]
for $z_0 \in \C$ and $t > 0$, assuming that the limit exists (possibly infinite).
\end{defn}

We use the following crucial result, proved in \cite{bowen}, establishing the existence of the pressure function and Bowen's formula for transcendental meromorphic maps.

\begin{thm}[{\cite[Theorems A and B]{bowen}}]\label{thm:pressure}
For every transcendental entire or meromorphic map $f$ in the class $\mathcal S$ and every $t > 0$ the topological pressure $P(f,t) = P(f,t,z_0)$ exists $($possibly equal to $+\infty)$ and is independent of $z_0 \in \C$ up to an exceptional set of Hausdorff dimension zero $($consisting of points quickly approximated by the forward orbits of singular values of $f)$. We have
\[
P(f, t) = P_{hyp}(f, t),
\] 
where $P_{hyp}(f, t)$ is the supremum of the pressures $P(f|_X, t)$
over all transitive isolated conformal repellers $X \subset J(f)$. The
function $t \mapsto P(f, t)$ is non-increasing and convex when it is
finite and satisfies $P(f,2) \leq 0$.
The following version of Bowen's formula holds:
\[
\dim_H J_r(f)  = \dim_{hyp} J(f) = t_0,
\]
where $t_0 = \inf \{t > 0: P(f,t) \leq 0\}$.

Moreover, analogous results hold for every non-exceptional transcendental entire or meromorphic map $f$ in the class $\B$, such that $J(f) \setminus \overline{\mathcal P(f)} \neq \emptyset$ $($in particular, for every hyperbolic map$)$ and $z$ in $J(f) \setminus \overline{\mathcal P(f)}$, which is an open dense
subset of $J(f)$.

If $f$ is hyperbolic, then $P(f, t) > 0$ $($possibly equal to $+\infty)$ for every $0 < t < t_0$ and $P(f, t) < 0$ for every $t > t_0$.
\end{thm}

We refer to the points $z_0$ from Theorem~\ref{thm:pressure} as to GPS points (\emph{good pressure starting points}). 

Following \cite{bowen}, we denote
\[
S_n(t,z) = \sum_{w\in f^{-n}(z)}\frac{1}{|(f^n)^*(w)|^t}, \qquad
S_n^A(t,z) = \sum_{w\in f^{-n}(z) \cap A}\frac{1}{|(f^n)^*(w)|^t}
\]
for a set $A \subset \C$ and $n > 0$.

\section{Introductory results}\label{sec:intro_results}

In this section we present some results concerning the properties of the radial Julia set, conformal measures and topological pressure. Note that some facts are not used in the subsequent sections but are of independent interest. In all results of this section we assume that $f$ is a transcendental meromorphic map on the complex plane.

The first two propositions study properties of the radial Julia set $J_r(f)$ and the set of non-escaping points in the Julia set. 

\begin{prop}\label{prop:J_r}

\

\begin{itemize}
\item[$(a)$]
If $f$ has a finite number of poles, then $J_r(f) \subset J(f) \setminus (I(f) \cup \bigcup_{n=1}^\infty f^{-n}(\infty))$.\\In particular, if $f$ is entire, then $J_r(f) \subset J(f) \setminus I(f)$.
\item[$(b)$] If $f$ is hyperbolic, then $J(f) \setminus (I(f) \cup \bigcup_{n=1}^\infty f^{-n}(\infty)) \subset J_r(f)$.\\
In particular, if $f$ is hyperbolic entire, then $J_r(f) = J(f) \setminus I(f)$.
\end{itemize}
\end{prop}
\begin{proof} First, note that by definition, $J_r(f) \cup (I(f)\cap J(f)) \subset J(f) \setminus \bigcup_{n=1}^\infty f^{-n}(\infty)$.

To show (a), assume that $f$ has a finite number of poles and suppose that $z \in J_r(f) \cap I(f)$. Then there exists a sequence $n_k \to \infty$ and $r > 0$, such that the branch $g_k$ of $f^{-n_k}$ sending $f^{n_k}(z)$ to $f^{n_k-1}(z)$ is defined on $\DD(f^{n_k}(z), r)$. Since $z \in I(f)$, we have $\infty \in \DD(f^{n_k}(z), r/2)$ for large $k$, which implies that $g_k$ is defined on $\DD(\infty, r/2)$, sending $\infty$ to some pole $p_k$ of $f$. Since the number of poles is finite, passing to a subsequence we can assume that $p_k \equiv p$. Then $f^{n_k-1}(z) = g_k(f^{n_k}(z))$ is in a small neighbourhood of the pole $p$ for all large $k$, which contradicts $z \in I(f)$. 

To prove (b), suppose that $f$ is hyperbolic and take $z \in J(f) \setminus (I(f) \cup \bigcup_{n=1}^\infty f^{-n}(\infty))$. Then there exist a sequence $n_k \to \infty$ and $R > 0$, such that $|f^{n_k}(z)| < R$ for all $k$. By the definition of hyperbolicity and the fact that the spherical and Euclidean metric are comparable on compact sets in $\C$ we conclude that there exists $r > 0$ such that $\DD(f^{n_k}(z), r) \cap \overline{\PP(f)} = \emptyset$ for all $k$, which gives $z \in J_r(f)$. 
\end{proof}

\begin{prop}\label{prop:bd} 

\

\begin{itemize}
\item[$(a)$] If $J(f)\neq \C$ $($which holds if and only if $J(f)$ has empty interior$)$, then the $2$-dimensional Lebesgue measure of $J_r(f)$ is zero.

\item[$(b)$] If $f$ is hyperbolic, then the $2$-dimensional Lebesgue measure of $J(f) \setminus I(f)$ is zero.
\end{itemize}
\end{prop}
\begin{proof} To prove the assertion (a), suppose the $2$-dimensional Lebesgue measure of $J_r(f)$ is positive and take a density point $z \in J_r(f)$. Then there exists a sequence $n_k \to \infty$ and $r > 0$, such that the branch $g_k$ of $f^{-n_k}$ sending $f^{n_k}(z)$ to $z$ is defined on $\DD(f^{n_k}(z), r)$. Let $D_k = \DD(f^{n_k}(z), r/2)$. Then $\diam_{sph} D_k \le \diam_{sph} \clC / 2$ and $\clC \setminus g_k(D_k)$ contains the set $\Sing(f) \cup \{\infty\}$, which has at least two elements and hence positive spherical diameter. Therefore, by Theorem~\ref{koebe}, the distortion of $g_k$ on $D_k$ is bounded by a constant independent of $k$. Since $z \in J(f)$, a standard normality argument (see e.g. the proof of \cite[Lemma 3.12]{bowen}) gives 
\[
d_k = \diam g_k(D_k) \to 0 \quad \text{as} \quad k \to \infty.
\]
As $J(f)$ has empty interior, there exists an $\varepsilon > 0$, such that every spherical disc of radius $r/2$ contains a spherical disc of radius $\varepsilon$, which is disjoint from $J(f)$. Consequently, by the bounded distortion of $g_k$ on $D_k$ and since the spherical and Euclidean metric are comparable on compact sets in $\C$, the disc $\D(z, d_k)$ contains a disc of radius $c d_k$ disjoint from $J(f)$ (in particular, from $J_r(f)$), for some fixed $c > 0$. Since $d_k \to 0$, this contradicts the fact that $z$ is the density point of $J_r(f)$ and proves~(a).

To proof the assertion (b), it is enough to use (a) and Proposition~\ref{prop:J_r}.
\end{proof}

The next proposition describes some properties of conformal measures.

\begin{prop}\label{prop:support}
If $m_t$ is a $t$-conformal measure for some meromorphic map $f$ and $t > 0$, then either $\nu$ is positive on open sets in $J(f)$ or $m_t$ is supported on the set of $($at most two$)$ Picard exceptional values of $f$. 
\end{prop}
\begin{proof}
Let $z\in J(f)$ and $B=\D(z,r)$ for some $r>0$. We consider separately two cases. First, assume that all iterates $f^n$ are defined in $B$. Since the family $\{f^n|_B\}_{n>0}$ is not normal, the union $\bigcup_{n >0} f^n(B)$ covers the whole sphere $\clC$ with at most two points $\{p_1,p_2\}$  excluded. If $m_t(B)=0$ then $m_t(f^n(B))=0$ for every $n>0$. This fact follows easily from the formula \eqref{conformality} and from the fact that the disc $B$ can be divided into a finite union of Borel subsets on which the map $f^n$ is injective.

In the second case, $f^n(B)$ contains some pole of the function $f$, so $f^{n+1}(B)$ contains a neighbourhood of $\infty$. By Picard's/Nevanlinna Theorem applied for the meromorphic function $f$, it assumes every value, with possibly two exceptions $p_1,p_2\in\clC$, infinitely many times in every neighbourhood of $\infty$. This implies that $f^{n+2}(B)\supset\clC\setminus\{p_1,p_2\}$.

In both cases we conclude that the measure $\nu$ would be supported on at most two points $p_1, p_2\in \C$. 
Again by \eqref{conformality}, the set $\{p_1,p_2\}$ would be backward invariant. Hence, either $f(p_1)=p_1$ and $p_1$ has no other preimages or $f^2(p_2)=f(p_1)=p_2$ and, again, these points have no other preimages. Therefore, the measure $m_t$ would be supported on the set of (at most two) exceptional values of $f$.
\end{proof}

In fact, the second alternative can occur, as noted in the following example.

\begin{ex}
For $f(z) = ze^z$, the value $0$ is the unique finite exceptional value of $f$, with $f^{-1}(0) = \{0\}$, $f(0) = 0$ and $f'(0) = 1$. Consequently, $0 \in J(f)$ and the Dirac measure $\delta_0$ is $t$-conformal for every $t > 0$.
\end{ex}

The next propositions consider properties of the topological pressure. The first one is an essential strengthening of \cite[Proposition~5.7]{bowen}.

\begin{prop}\label{strange} Let $f$ satisfy the assumptions of Theorem~{\rm \ref{thm:pressure}}. Then for every $t > 0$, every GPS point $z_0$ and sufficiently large $r > 0$ we have
\[
P(f, t)=\limsup_{n\to\infty}\frac{1}{n}\ln S^{\D(r)}_n (t, z_0).
\]
\end{prop}

\begin{proof}
It follows from \cite{bowen} that 
\begin{equation}\label{limsup}
P(f,t)=\sup_{r > 0} \limsup_{n\to\infty}\frac{1}{n}\ln S^{\D(r)}_n (t, z_0).
\end{equation}
However, by Lemma~\ref{lem:compact2} (proved in Section~\ref{sec:C}), we have 
\[
S_n^{\D(2r)}(t,z_0)=S_{n}^{\D(2r)\setminus \D(r)}(t,z_0)+S_n^{\D(r)}(t,z_0)\le c\frac{(\ln r)^{3t}}{r^t} S_{n+1}^{\D(r)}(t,z_0)+ S_n^{\D(r)}(t,z_0)
\]
for sufficiently large $r$ and a constant $c > 0$. Hence, $\limsup$ in the formula \eqref{limsup} is the same for $r$ and $2r$. Since it is non-decreasing with $r$, this implies that it is actually constant for large $r$, which ends the proof.
\end{proof} 

The next result describes the alternative related to the existence of a zero of the pressure function. 

\begin{prop}\label{prop:clasif} Let $f$ satisfy the assumptions of Theorem~{\rm \ref{thm:pressure}} and let
\[
t_0 = \inf\{t > 0: P(f, t) \le 0\}, \qquad t_\infty = \sup\{t \ge 0: P(f, t) = + \infty\}.
\]
Then one of the three possibilities occurs:
\begin{itemize}
\item[$(a)$] $P(f, t_\infty) = +\infty$, $\lim_{t \to t_\infty^+} P(f,t) = +\infty$, $t_0 > t_\infty$ and $P(f, t_0) = 0$. 
\item[$(b)$] $0 \le P(f, t_\infty) < +\infty$, $t_0 \ge t_\infty$ and $P(f, t_0) = 0$. 
\item[$(c)$] $t_0 = t_\infty$ and $P(f, t_\infty) = P(f,t_0) < 0$. 
\end{itemize}
\end{prop}
\begin{proof}Suppose $P(f, t_\infty) = +\infty$. By Theorem~\ref{thm:pressure}, $\sup_X P(f|_X, t_\infty) = \infty$, where $\sup$ is taken over all compact repellers $X \subset J(f)$. Since the pressure function $P(f|_X, t)$ is finite and continuous for all $t \ge 0$ (see e.g.~\cite{pu-book}), this implies $\lim_{t \to t_\infty^+} P(f,t) = +\infty$. Similarly we show that if $P(f, t_\infty) < \infty$, then $\lim_{t \to t_\infty^+} P(f,t) = P(f, t_\infty)$. Other assertions of the proposition follows easily form the fact that the function $t \mapsto P(f, t)$ is non-increasing and convex (and hence continuous) for $t \in (t_\infty, +\infty)$ and $P(f, 2) \le 0$ (see Theorem~\ref{thm:pressure}).
\end{proof}

\begin{rem} By \cite{BKZ-IMRN}, we have $t_0 > 1$.
\end{rem}

Note that the cases (a) and (b) in Proposition~\ref{prop:clasif} correspond to the existence of a zero of the pressure function. It is an open problem, whether the case (c) can actually appear for a transcendental map satisfying the assumptions of Theorem~{\rm \ref{thm:pressure}}.

\section{Proof of Theorem B}\label{sec:B}
\noindent{\it Case}~(a)\nopagebreak

\smallskip

Assume that $f \in \mathcal S$ and take a $t$-conformal measure $m_t$ on $J(f)$, with respect to the spherical metric and some $t > 0$.
Let
\[
D_n = \DD\left(\bigcup_{m=0}^{n-1}f^m(\Sing(f)) \cup \{\infty\}, e^{-\sqrt{n}}\right),
\]
where we denote
\[
\DD (X, r) = \{z \in \clC: \dist_{sph}(z, x) < r \text{ for some } x \in X\},
\] 
and
\[
E = \bigcap_{k=1}^\infty \bigcup_{n = k}^\infty D_n.
\]
By definition, $E \subset \overline{\PP(f)}\cup\{\infty\}$ and it is easily seen that the Hausdorff dimension of $E$ equals $0$. Let
\[
A_k =  \bigcap_{n = k}^\infty (\clC \setminus D_n)
\]
for $k \ge 1$. 
Then $(A_k)_{k=1}^\infty$ forms an increasing sequence of sets, and 
\[
\clC\setminus E=\bigcup_{k=1}^\infty A_k.
\]
By the definition of $A_k$, the spherical area of $\clC \setminus  A_k  = \bigcup_{n \ge k}^\infty D_n$ is smaller than 
\[
(\#\Sing(f)+1)\sum_{n=k}^\infty n e^{-\sqrt{n}},
\] 
which tends $0$ as $k \to \infty$. Hence, for large $k$ the set $A_k$ has positive area and, in particular, there exists a GPS point $z_0 \in A_k$. 

Again by the definition of $A_k$, for every $z \in A_k$ all branches of $f^{-n}$, $n \ge k$ are defined on $\DD(z, e^{-\sqrt{n}})$. Moreover, for every $n \ge k$ there exists at most countable partition $\{A_{k, j}^{(n)}\}_{j}$ of $A_k$ by non-empty Borel sets, such that $A_{k, j}^{(n)} \subset \DD(v_{k,j}^{(n)}, e^{-\sqrt{n}}/2)$ 
for some $v_{k, j}^{(n)} \in A_k$. By Theorem~\ref{koebe}, the distortion of all branches of $f^{-n}$ on $A_{k, j}^{(n)}$ is bounded by a constant independent of $n$. This and the $t$-conformality of $m_t$ imply
\begin{equation}\label{eq:m}
\frac{m_t(A_{k, j}^{(n)})}{C}\sum_{w\in f^{-n}(v_{k, j}^{(n)})}\frac{1}{|(f^n)^*(w)|^t} \le
m_t(f^{-n}(A_{k, j}^{(n)})) \le C m_t(A_{k, j}^{(n)})\sum_{w\in f^{-n}(v_{k, j}^{(n)})}\frac{1}{|(f^n)^*(w)|^t}
\end{equation}
for some $C > 0$ independent of $k, n, j$.
Using the definition of $A_k$ and \cite[Lemma~3.1]{conical}, we show that there exists $c > 0$ such that
\begin{equation}\label{eq:sum}
\frac 1 {e^{c\,\sqrt{n}}} \sum_{w\in f^{-n}(z_0)}\frac{1}{|(f^n)^*(w)|^t} \leq \sum_{w\in f^{-n}(v_{k, j}^{(n)})}\frac{1}{|(f^n)^*(w)|^t} \leq e^{c\,\sqrt{n}} \sum_{w\in f^{-n}(z_0)}\frac{1}{|(f^n)^*(w)|^t}
\end{equation}
for every $n \ge k$ and every set $A_{k, j}^{(n)}$ (see the proof of \cite[Lemma~5.4]{bowen} for details).

To prove the first assertion of the theorem, suppose that $m_t$ is not supported on $E$. Then 
\[
m_t(A_k) > 0
\]
for some $k$. Since $(A_k)_k$ is increasing, this holds for every sufficiently large $k$. Hence, by \eqref{eq:m} and \eqref{eq:sum}, we obtain
\begin{multline*}
m_t(f^{-n}(A_k)) = \sum_j m_t(f^{-n}(A_{k, j}^{(n)})) \ge \frac 1 C \sum_j m_t(A_{k, j}^{(n)})\sum_{w\in f^{-n}(v_{k, j}^{(n)})}\frac{1}{|(f^n)^*(w)|^t}\\ \ge \frac 1 C \sum_j \frac{m_t(A_{k, j}^{(n)})}{e^{c\,\sqrt{n}}}\sum_{w\in f^{-n}(z_0)}\frac{1}{|(f^n)^*(w)|^t}
= \frac{m_t(A_k)}{C e^{c\,\sqrt{n}}}\sum_{w\in f^{-n}(z_0)}\frac{1}{|(f^n)^*(w)|^t}.
\end{multline*}
This implies 
\[
\sum_{w\in f^{-n}(z_0)}\frac{1}{|(f^n)^*(w)|^t} \le \frac{C e^{c\,\sqrt{n}}}{m_t(A_k)} 
\]
for every $n \ge k$, which gives $P(t, f) \le 0$ and ends the proof of the first assertion.

\

To prove the second assertion, suppose $P(t, f) < 0$. Then
\[
\sum_{w\in f^{-n}(z_0)}\frac{1}{|(f^n)^*(w)|^t} < e^{-n\delta}
\]
for some $\delta > 0$ and every sufficiently large $n$.
Using again \eqref{eq:m} and \eqref{eq:sum}, we obtain, similarly as previously, for sufficiently large $n$, 
\begin{multline*}
m_t(f^{-n}(A_k))  \le C \sum_j m_t(A_{k, j}^{(n)})\sum_{w\in f^{-n}(v_{k, j}^{(n)})}\frac{1}{|(f^n)^*(w)|^t}\\
\le  Cm_t(A_k)e^{c\,\sqrt{n}}\sum_{w\in f^{-n}(z_0)}\frac{1}{|(f^n)^*(w)|^t} < C e^{c\,\sqrt{n}-n\delta} < e^{-n\delta/2}.
\end{multline*}
This shows that  for sufficiently large $k$, the series $\sum_n m_t(f^{-n}(A_k))$ is converging, so by the Borel--Cantelli Lemma, for $m_t$-almost every $z$ there exists $m_0=m_0(z)$ such that $f^m(z)\in \C \setminus A_k \subset \bigcup_{n = k}^\infty D_n$ for all $m\ge m_0$. This implies that for $m_t$-almost every point $z \in \C$, for every large $k$ there exists $m_0$ such that for every $m \ge m_0$ we have $\dist_{sph}(f^m(z), \zeta_m) < e^{-\sqrt{k}}$ for some $\zeta_m \in \PP(f) \cup \{\infty\}$. This proves the second assertion of (a).

\bigskip

\noindent{\it Case}~(b)\nopagebreak

\smallskip

Assume now $f \in \mathcal B$ is non-exceptional and $J(f) \setminus\overline{\PP(f)} \neq \emptyset$. The proof is analogous to the one in case~(a). Take an open spherical disc $D = \DD(z_0, r)$ for a small $r > 0$, such that $\DD(z_0, 2r) \cap (\overline{\PP(f)} \cup \{\infty\}) = \emptyset$. Then all branches of $f^{-n}$, $n > 0$ are defined on $D$, in particular $z_0$ is a GPS point. By Theorem~\ref{koebe}, the distortion of the branches is universally bounded. This together with the $t$-conformality of $m_t$ gives
\begin{equation}\label{eq:m'}
\frac{m_t(D)}{C}\sum_{w\in f^{-n}(z_0)}\frac{1}{|(f^n)^*(w)|^t} \le
m_t(f^{-n}(D)) \le C m_t(D)\sum_{w\in f^{-n}(z_0)}\frac{1}{|(f^n)^*(w)|^t}
\end{equation}
for some $C > 0$. 

Since, by Proposition~\ref{prop:support}, $m_t$ is not supported on $\overline{\PP(f)}$, there exists a disc $D= \DD(z_0, r)$ as above with $m_t(D) > 0$, so \eqref{eq:m'} gives 
\[
\sum_{w\in f^{-n}(z_0)}\frac{1}{|(f^n)^*(w)|^t} \le \frac{C m_t(f^{-n}(D))}{m_t(D)}
\]
for every $n$, which implies $P(f, t) \le 0$. This proves the first assertion. 

Suppose now $P(f, t) < 0$. Then there exists $\delta > 0$, such that 
\[
\sum_{w\in f^{-n}(z_0)}\frac{1}{|(f^n)^*(w)|^t} < e^{-n\delta}
\]
for large $n$, so \eqref{eq:m'} implies
\[
m_t(f^{-n}(D)) < Cm_t(\C) e^{-n\delta}
\]
for every disc $D$ as above. Again, by the Borel--Cantelli Lemma, for $m_t$-almost every $z$ there exists $n_0$ such that $f^n(z)\in \C \setminus D$ for all $n\ge n_0$. Since for every $\varepsilon > 0$ the set $\C \setminus \DD(\overline{\PP(f)} \cup \{\infty\}, \varepsilon)$ can be covered
by a finite number of discs $D$ as above, we conclude that for $m_t$-almost every point $z \in \C$ we have $f^n(z) \to \overline{\PP(f)} \cup \{\infty\}$ as $n \to \infty$. This ends the proof of the second assertion of (b).

\section{Construction of a conformal measure -- proof of Theorem~C}\label{sec:C}

Throughout this section, we assume that $f$ is a map from class $\SSS$ with a logarithmic tract over $\infty$ or a (non-exceptional) map from class $\B$, with a logarithmic tract over $\infty$ and $J(f) \setminus\overline{\PP(f)} \neq \emptyset$.

To prove Theorem~C, we need the following lemmas. 
Fix $t > 0$.

\begin{lem}\label{lem:v}
There exist $c_1, c_2, r_0>0$ such that for every for every $z \in\C$ with $|z| \ge r_0$, 
\[
S_1^{\D(c_1(\ln|z|)^{4\pi})}(t,z)>\frac{c_2|z|^t}{(\ln |z|)^{3t}}.
\]
\end{lem}
\begin{proof} By assumption, $f$ has a logarithmic tract $U$ over $\infty$. 
Fix $w\in\C$ of large modulus and let $g$ be a branch of $f^{-1}$ defined in a neighbourhood of $v$ leading to $U$. Take an arbitrary point $z$ with $|z|>|w|$. By Lemma~\ref{lem:|g^*|}, applied for $z_1 = z, z_2 = w$, we have 
\[
|g^*(z)| > \frac{c|z|}{(\ln |z|)^3}
\]
for some extension of the branch $g$, where $c$ is a constant depending only on $w$.
Similarly, by Lemma~\ref{lem:|g|}, there exists a constant $c_2 > 0$, such that
\[
|g(z)| < c_1 (\ln|z|)^{4\pi} 
\]
if $|z|$ is sufficiently large. We conclude that
\[
S_1^{\D( c_1(\ln|z|)^{4\pi})}(t,z) \ge |g^*(z)|^t >  \frac{c^t|z|^t}{(\ln |z|)^{3t}}.
\]
\end{proof}
\begin{cor}\label{cor:v}
There exist $c, r_0>0$ such that for every for every $z \in\C$ with $|z| \ge r_0$, 
\[
S_1(t,z)>\frac{c|z|^t}{(\ln |z|)^{3t}}.
\]
\end{cor}

\begin{lem}\label{lem:compact} Let $z_0 \in J(f)$ be a GPS point. Then there exist $c, k_0 > 0$ such that for every $n \ge 1$ there holds
\[
S_{n+1}(t,z_0) \ge c \sum_{k = k_0}^\infty \frac{2^{kt}}{k^{3t}}S^{\D(2^{k+1})\setminus \D(2^k)}_n(t,z_0).
\]
\end{lem}
\begin{proof}
By Corollary~\ref{cor:v} and the fact that the function $x \mapsto x^t/(\ln x)^{3t}$ is increasing for large $x > 0$, for sufficiently large $k_0$ we obtain
\begin{align*}
S_{n+1}(t,z_0)&=\sum_{w\in f^{-(n+1)}(z_0)}\frac{1}{|(f^{n+1})^*(w)|^t}=
\sum_{z\in f^{-n}(z_0)}
\frac{S_1(t,z)}{|(f^n)^*(z)|^t}\\
&\ge \sum_{k = k_0}^\infty \sum_{\substack{z\in f^{-n}(z_0),\\ z\in \D(2^{k+1})\setminus \D(2^k)}}\frac{S_1(t,z)}{|(f^n)^*(z)|^t}
\ge\frac{c}{(\ln 2)^{3t}} \sum_{k = k_0}^\infty \frac{2^{kt}}{k^{3t}} S^{\D(2^{k+1})\setminus \D(2^k)}_n(t,z_0).
\end{align*}
\end{proof}
Using more detailed  estimates provided by Lemma~\ref{lem:v} instead of Corollary~\ref{cor:v}, we obtain the following, slightly  more delicate result, which is needed in the proof of Proposition~\ref{strange}.

\begin{lem}\label{lem:compact2} Let $z_0 \in J(f)$ be a GPS point. Then there exist $\tilde c_1, \tilde c_2, r_0 > 0$ such that for every $r > r_0$ and $n \ge 1$ there holds
\[
S_{n+1}^{\D(\tilde c_1(\ln r)^{4\pi})}(t,z_0) \ge \frac{\tilde c_2r^t}{(\ln r)^{3t}} S^{\D(2r)\setminus \D(r)}_n(t,z_0).
\]
\end{lem}
\begin{proof} Analogously as in the proof of Lemma~\ref{lem:compact}, using Lemma~\ref{lem:v} and the monotonicity of $x \mapsto x^t/(\ln x)^{3t}$ for large $x > 0$, we obtain
\begin{align*}
S_{n+1}^{\D(c_1(\ln 2r)^{4\pi})}(t,z_0)&=
\sum_{z\in f^{-n}(z_0)}
\frac{S_1^{\D(c_1(\ln r)^{4\pi})}(t,z)}{|(f^n)^*(z)|^t}\ge 
\sum_{\substack{z\in f^{-n}(z_0),\\ z\in \D(2r) \setminus \D(r)}}\frac{S_1^{\D(c_1(\ln 2r)^{4\pi})}(t,z)}{|(f^n)^*(z)|^t}\\
&\ge\frac{c_2r^t}{(\ln 2r)^{3t}}\sum_{\substack{z\in f^{-n}(z_0),\\ z\in \D(2r) \setminus \D(r)}}\frac{1}{|(f^n)^*(z)|^t} = \frac{c_2r^t}{(\ln 2r)^{3t}} S^{\D(2r)\setminus \D(r)}_n(t,z_0)
\end{align*}
for sufficiently large $r$.
\end{proof}

Assume that there exists $t>0$ such that $P(f,t)=0$ and fix this value of $t$ from now on. As in \cite{denker-urbanski,mayer-urbanski-etds}, we consider the probability measures
\[
\mu_s = \frac{1}{\Sigma_s}\sum_{n=1}^\infty b_n e^{-ns}\sum_{w\in f^{-n}(z_0)}\frac{\delta_w}{|(f^n)^*(w)|^t},
\]
where $s > 0$, $z_0 \in J(f)$ is a GPS point, $b_n$ is a sequence of positive real numbers (independent of $s$), $\delta_w$ denotes the Dirac measure at $w$, and
\[
\Sigma_s = \sum_{n=1}^\infty b_n e^{-ns} S_n(t, z_0).
\]
Note that we can choose a GPS point $z_0 \in J(f)$, since the Hausdorff dimension of $J(f)$ is always positive (see \cite{stallard-dim}).
Since 
\[
P(f,t) = \lim_{n\to \infty} \frac 1 n \ln S_n(t, z_0) = 0,
\]
the series
\[
\sum_{n=1}^\infty e^{-ns}S_n(t, z_0)
\]
is convergent for $s > 0$. Moreover, we can choose the sequence $b_n$ so that
\begin{equation}\label{eq:b_n}
\lim_{n\to \infty} \frac{b_{n+1}}{b_n} = 1,  \qquad \lim_{s \to 0^+} \Sigma_s = +\infty
\end{equation}
(see \cite[Lemma 3.1]{denker-urbanski}). In particular, the measure $\mu_s$ is well-defined, with its support in $\C$. Since $z_0 \in J(f)$, in fact $\mu_s$ is supported on $J(f)$.

\begin{lem}\label{almost_tight}
There exists a constant $c > 0$ such that for every $0 < s < 1$,
\[
\sum_{k = 1}^\infty \frac{2^{kt}}{k^{3t}} \mu_s(\D(2^{k+1})\setminus \D(2^k)) < c.
\]
\end{lem}

\begin{proof}
Write 
\[
\nu_n = \sum_{w\in f^{-n}(z_0)}\frac{\delta_w}{|(f^n)^*(w)|^t},
\]
Then 
\[
\mu_s=\frac{1}{\Sigma_s}\sum_{n=1}^\infty b_n e^{-ns}\nu_n.
\]
By Lemma~\ref{lem:compact},
\[
\nu_{n+1}(J(f)) =  
S_{n+1}(t,z_0) \ge c \sum_{k = k_0}^\infty \frac{2^{kt}}{k^{3t}} S^{\D(2^{k+1})\setminus \D(2^k)}_n(t,z_0) = 
c \sum_{k = k_0}^\infty \frac{2^{kt}}{k^{3t}} \nu_n(\D(2^{k+1})\setminus \D(2^k)).
\]
This together with \eqref{eq:b_n} implies
\begin{align*}
\sum_{k = k_0}^\infty \frac{2^{kt}}{k^{3t}} \mu_s(\D(2^{k+1})\setminus \D(2^k)) &= \frac{1}{\Sigma_s}\sum_{k = k_0}^\infty\frac{2^{kt}}{k^{3t}}\sum_{n=1}^\infty b_n e^{-ns}\nu_n(\D(2^{k+1})\setminus \D(2^k))\\
&\le \frac{1}{c\Sigma_s} \sum_{n=1}^\infty b_n e^{-ns}\nu_{n+1}(J(f))\\
&\leq  \frac{e^s}{c}\sup_n \frac{b_n}{b_{n+1}}  \frac{1}{\Sigma_s} \sum_{n=1}^\infty b_{n+1} e^{-(n+1)s}\nu_{n+1}(J(f))\\
&= \frac{e^s}{c}\sup_n \frac{b_n}{b_{n+1}}  \left(\mu_s(J(f)) - \frac{b_1 e^{-s}}{\Sigma_s}\nu_1(J(f))\right)\\
&< \frac{e}{c}\sup_n \frac{b_n}{b_{n+1}}  \mu_s(J(f))\\
&= \frac{e}{c}\sup_n \frac{b_n}{b_{n+1}}.
\end{align*}
To end the proof, note that
\[
\sum_{k = 1}^{k_0-1} \frac{2^{kt}}{k^{3t}} \mu_s(\D(2^{k+1})\setminus \D(2^k)) \le \sum_{k = 1}^{k_0-1} \frac{2^{kt}}{k^{3t}} < \frac{2^{k_0t}}{2^t-1}.
\]

\end{proof}

\begin{prop}\label{tight}
There exists a constant $c > 0$ such that for every $0 < s < 1$,
\[
\sum_{k = 1}^\infty \frac{2^{kt}}{k^{3t}} \mu_s(J(f)\setminus \D(2^k)) < c.
\]
\end{prop}

\begin{proof} Enlarging $k_0$ from the previous lemmas, we can assume
\[
\frac{2^{(k-1)t}}{(k-1)^{3t}} < q\frac{2^{kt}}{k^{3t}}
\]
for some fixed $0 < q < 1$ and for every $k \ge k_0$. Hence,
\[
\sum_{k = k_0}^j\frac{2^{kt}}{k^{3t}} < \frac{1}{1-q} \frac{2^{jt}}{j^{3t}}.
\]
for every $j \ge k_0$.
Using this and Lemma~\ref{almost_tight}, we obtain
\begin{align*}
\sum_{k = k_0}^\infty \frac{2^{kt}}{k^{3t}} \mu_s(J(f)\setminus \D(2^k))&=
\sum_{k = k_0}^\infty \sum_{j = k}^\infty\frac{2^{kt}}{k^{3t}}  \mu_s(\D(2^{j+1})\setminus \D(2^j))\\
&=\sum_{j = k_0}^\infty \sum_{k = k_0}^j\frac{2^{kt}}{k^{3t}}  \mu_s(\D(2^{j+1})\setminus \D(2^j))\\
&<\frac{1}{1-q}\sum_{j = k_0}^\infty \frac{2^{jt}}{j^{3t}}  \mu_s(\D(2^{j+1})\setminus \D(2^j)) < c.
\end{align*}
Since 
\[
\sum_{k = 1}^{k_0-1} \frac{2^{kt}}{k^{3t}} \mu_s(J(f)\setminus \D(2^k)) \le \sum_{k = 1}^{k_0-1} \frac{2^{kt}}{k^{3t}} < \frac{2^{k_0t}}{2^t-1},
\]
this ends the proof of the proposition.
\end{proof}

\begin{cor}\label{cor:tight}
The family $\{\mu_s\}_{s \in (0, 1)}$ is tight. Consequently, there exists a weak limit
\[
m_t = \lim_{j \to \infty} \mu_{s_j} 
\]
for some sequence $s_j \to 0^+$, which is a probability measure with support in $J(f)$.
\end{cor}
\begin{proof} To show the tightness of the family $\{\mu_s\}_{s \in (0, 1)}$, it is sufficient to check that for every $\varepsilon > 0$ there exists a compact subset $K$ of $J(f)$ such that $\mu_s(K) > 1 - \varepsilon$ for every $0<s<1$.  
This follows immediate from the estimation
\[
\mu_s(J(f)\setminus \D(2^k)) < c\frac{k^{3t}}{2^{kt}}
\]
for every $0<s<1$, which is a consequence of Proposition~\ref{tight}.
\end{proof}

\begin{cor}\label{cor:small}
There exists a constant $c > 0$ such that
\[
\sum_{k = 1}^\infty \frac{2^{kt}}{k^{3t}} m_t(J(f)\setminus \D(2^k)) < c.
\]
\end{cor}
\begin{proof} The estimate follows easily from Proposition~\ref{tight} and the fact 
\[
m_t(U) \le \liminf_{j \to \infty}\mu_{s_j}(U)
\]
for every open subset $U$ of $J(f)$.
\end{proof}

Note that Corollary~\ref{cor:small} proves Remark~\ref{rem:thmB} immediately.

\begin{prop}\label{prop:conf}
The measure $m_t$ is $t$-conformal with respect to the spherical metric.
\end{prop}
\begin{proof}
The proof is the same as in \cite{denker-urbanski, mayer-urbanski-etds}. Take a set $A \subset \C$ such that $f$ is univalent on $A$. By definition, we have
\[
\mu_{s_j}(f(A)) = \frac{1}{\Sigma_{s_j}} \sum_{n=1}^\infty b_n e^{-ns_j} S^{f(A)}_n(t,z_0)
\]
and
\begin{align*}
\int_A |f^*|^t d \mu_{s_j} &= \frac{1}{\Sigma_{s_j}} \sum_{n=1}^\infty b_n e^{-ns_j} \sum_{w \in f^{-n}(z_0) \cap A} \frac{|f^*(w)|^t}{|(f^n)^*(w)|^t} \\&= \frac{b_1 e^{-s_j}}{\Sigma_{s_j}} + 
\frac{1}{\Sigma_{s_j}} \sum_{n=2}^\infty b_n e^{-ns_j} S^{f(A)}_{n-1}(t,z_0)\\
&= \frac{b_1e^{-s_j}}{\Sigma_{s_j}} + \frac{1}{\Sigma_{s_j}} \sum_{n=1}^\infty b_{n+1} e^{-(n+1)s_j} S^{f(A)}_n(t,z_0).
\end{align*}
Take a small $\varepsilon > 0$ and fix $n_0$ such that $|b_{n+1}/b_n - 1| < \varepsilon$ for every $n \ge n_0$. Then
\begin{multline*}
\left|\int_A |f^*|^t d \mu_{s_j} - \mu_{s_j}(f(A))\right| \le \\ \frac{b_1 e^{-s_j}}{\Sigma_{s_j}} + \frac{1}{\Sigma_{s_j}}\sum_{n = 1}^{n_0 - 1}(b_{n+1} + b_n) S^{f(A)}_n(t,z_0) + \frac{1}{\Sigma_{s_j}}\sum_{n = n_0}^\infty|b_{n+1}e^{-s_j} - b_n| e^{-ns_j}S^{f(A)}_n(t,z_0).
\end{multline*}
The first and second term in the latter formula tend to $0$ as $j \to \infty$, since $\Sigma_{s_j} \to \infty$ (see \eqref{eq:b_n}). The third term can be estimated as
\begin{multline*}
\frac{1}{\Sigma_{s_j}}\sum_{n = n_0}^\infty|b_{n+1}e^{-s_j} - b_n| e^{-ns_j}S^{f(A)}_n(t,z_0) = \frac{1}{\Sigma_{s_j}}\sum_{n = n_0}^\infty\left|e^{-s_j}\frac{b_{n+1}}{b_n} - 1\right| b_n e^{-ns_j}S^{f(A)}_n(t,z_0)\\
\le \frac{(1 + e^{-s_j)}\varepsilon}{\Sigma_{s_j}}\sum_{n = n_0}^\infty b_n e^{-ns_j}S^{f(A)}_n(t,z_0) \le (1 + e^{-s_j})\varepsilon \mu_{s_j}(f(A)) \le 2 \varepsilon.
\end{multline*}
We conclude that 
\[
\left|\int_A |f^*|^t d \mu_{s_j} - \mu_{s_j}(f(A))\right| \to 0
\]
as $j \to \infty$. This together with that fact
\[
m_t(U) \le \liminf_{j \to \infty}\mu_{s_j}(U), \qquad m_t(F) \ge \limsup_{j \to \infty}\mu_{s_j}(F)
\]
for every open subset $U$ and closed subset $F$ of $J(f)$ easily proves
\[
m_t(f(A)) = \int_A |f^*|^t d m_t.
\]
\end{proof}

The above proposition ends the proof of Theorem~C. We complete the paper by the following observation considering the case $t = 2$.

\begin{prop}\label{prop:=2} 
Suppose $P(f,2) = 0$ and let $m_2$ be the measure constructed in the proof of Theorem~{\rm C}. Then the following hold.

\begin{itemize} 
\item[$(a)$] If $J(f) \neq \C$, then $m_2(J_r(f)) = 0$.

\item[$(b)$] If $f$ is hyperbolic, then $m_2(J(f) \setminus (I(f)\cup \bigcup_{n=1}^\infty f^{-n}(\infty))) = 0$.
\end{itemize}
\end{prop}

\begin{proof} Let $\Leb_2$ denote the $2$-dimensional Lebesgue measure on $\clC$,  with respect to the spherical metric. To prove~(a), we will show
\begin{equation}\label{eq:delta}
\liminf_{\delta \to 0^+} \frac{m_2(\D(z, \delta))}{\delta^2} < \infty
\end{equation}
for $z \in J_r(f)$. This will imply that $m_2$ on $J_r(f)$ is absolutely continuous with respect to $\Leb_2$ (see e.g.~\cite[Theorem 2.12]{mattila}) and to prove (a), it will be enough to use Proposition~\ref{prop:bd}.

To show \eqref{eq:delta}, take $z \in J_r(f)$ and define the sequence $n_k$, branch $g_k$ of $f^{-n_k}$ on the spherical disk $D_k$ and numbers $r, d_k$ as in the proof of Proposition~\ref{prop:bd}. By the bounded distortion of $g_k$, we have
\[
\DD(z, c_1 d_k) \subset g_k(D_k)
\]
for some constant $c_1 > 0$. Moreover, $\Leb_2(D_k) \ge c_2$ for a constant $c_2 > 0$ depending only on $r$ (and thus only on $z$), and $m_2(D_k) \le 1$. Recall also that $d_k \to 0$ as $k \to \infty$. In view of these, using the $2$-conformality of $m_2$ and $\Leb_2$ (with respect to the spherical metric) together with the bounded distortion of $g_k$ and the comparability of the spherical and Euclidean metric on compact sets in $\C$, we obtain
\[
\frac{m_2(\D(z, c_1 d_k))}{d_k^2} \le c_3\frac{m_2(g_k(D_k))}{\Leb_2(g_k(D_k))} \le c_4 \frac{m_2(D_k)}{\Leb_2(D_k)} \le \frac{c_4}{c_2}
\]
for some constants $c_3, c_4$ independent of $k$, which shows \eqref{eq:delta} and ends the proof of~(a).

The assertion~(b) follows immediately from~(a) and Proposition~\ref{prop:J_r}.
\end{proof}

\bibliographystyle{amsalpha}
\bibliography{conformal}

\providecommand{\bysame}{\leavevmode\hbox to3em{\hrulefill}\thinspace}
\providecommand{\MR}{\relax\ifhmode\unskip\space\fi MR }
\providecommand{\MRhref}[2]{%
  \href{http://www.ams.org/mathscinet-getitem?mr=#1}{#2}
}
\providecommand{\href}[2]{#2}
\begin{thebibliography}{PRLS04}

\bibitem[Bar95]{baranski-tangent}
Krzysztof Bara{\'n}ski, \emph{Hausdorff dimension and measures on {J}ulia sets
  of some meromorphic maps}, Fund. Math. \textbf{147} (1995), no.~3, 239--260.
  \MR{1348721}

\bibitem[Ber93]{bergweiler}
Walter Bergweiler, \emph{Iteration of meromorphic functions}, Bull. Amer. Math.
  Soc. (N.S.) \textbf{29} (1993), no.~2, 151--188. \MR{1216719}

\bibitem[BKZ09]{BKZ-IMRN}
Krzysztof Bara{\'n}ski, Bogus{\l}awa Karpi{\'n}ska, and Anna Zdunik,
  \emph{Hyperbolic dimension of {J}ulia sets of meromorphic maps with
  logarithmic tracts}, Int. Math. Res. Not. IMRN (2009), no.~4, 615--624.
  \MR{2480096}

\bibitem[BKZ12]{bowen}
\bysame, \emph{Bowen's formula for meromorphic functions}, Ergodic Theory
  Dynam. Systems \textbf{32} (2012), no.~4, 1165--1189. \MR{2955309}

\bibitem[Bow79]{rufusbowen}
Rufus Bowen, \emph{Hausdorff dimension of quasicircles}, Inst. Hautes \'Etudes
  Sci. Publ. Math. (1979), no.~50, 11--25. \MR{556580}

\bibitem[BPS01]{BPS}
J{\'e}r{\^o}me Buzzi, Fr{\'e}d{\'e}ric Paccaut, and Bernard Schmitt,
  \emph{Conformal measures for multidimensional piecewise invertible maps},
  Ergodic Theory Dynam. Systems \textbf{21} (2001), no.~4, 1035--1049.
  \MR{1849600}

\bibitem[CS07]{skorulski}
Ion Coiculescu and Bart{\l}omiej Skorulski, \emph{Thermodynamic formalism of
  transcendental entire maps of finite singular type}, Monatsh. Math.
  \textbf{152} (2007), no.~2, 105--123. \MR{2346428}

\bibitem[DG99]{DG}
Manfred Denker and Mikhail Gordin, \emph{Gibbs measures for fibred systems},
  Adv. Math. \textbf{148} (1999), no.~2, 161--192. \MR{1736956}

\bibitem[DU91a]{DU-parab}
M.~Denker and M.~Urba{\'n}ski, \emph{Hausdorff and conformal measures on
  {J}ulia sets with a rationally indifferent periodic point}, J. London Math.
  Soc. (2) \textbf{43} (1991), no.~1, 107--118. \MR{1099090}

\bibitem[DU91b]{DU-conf}
\bysame, \emph{On {S}ullivan's conformal measures for rational maps of the
  {R}iemann sphere}, Nonlinearity \textbf{4} (1991), no.~2, 365--384.
  \MR{1107011}

\bibitem[DU91c]{denker-urbanski}
Manfred Denker and Mariusz Urba{\'n}ski, \emph{On the existence of conformal
  measures}, Trans. Amer. Math. Soc. \textbf{328} (1991), no.~2, 563--587.
  \MR{1014246 (92k:58155)}

\bibitem[GS09]{graczyk-smirnov}
Jacek Graczyk and Stanislav Smirnov, \emph{Non-uniform hyperbolicity in complex
  dynamics}, Invent. Math. \textbf{175} (2009), no.~2, 335--415. \MR{2470110}

\bibitem[KU02]{KU1}
Janina Kotus and Mariusz Urba{\'n}ski, \emph{Conformal, geometric and invariant
  measures for transcendental expanding functions}, Math. Ann. \textbf{324}
  (2002), no.~3, 619--656. \MR{1938460}

\bibitem[KU04]{KU2}
\bysame, \emph{Geometry and ergodic theory of non-recurrent elliptic
  functions}, J. Anal. Math. \textbf{93} (2004), 35--102. \MR{2110325}

\bibitem[KU05]{KU3}
\bysame, \emph{The dynamics and geometry of the {F}atou functions}, Discrete
  Contin. Dyn. Syst. \textbf{13} (2005), no.~2, 291--338. \MR{2152392}

\bibitem[KU06]{KU4}
\bysame, \emph{Geometry and dynamics of some meromorphic functions}, Math.
  Nachr. \textbf{279} (2006), no.~13-14, 1565--1584. \MR{2269255}

\bibitem[KU08]{ku}
\bysame, \emph{Fractal measures and ergodic theory of transcendental
  meromorphic functions}, Transcendental dynamics and complex analysis, London
  Math. Soc. Lecture Note Ser., vol. 348, Cambridge Univ. Press, Cambridge,
  2008, pp.~251--316. \MR{2458807}

\bibitem[LSV98]{LSS}
Carlangelo Liverani, Benoit Saussol, and Sandro Vaienti, \emph{Conformal
  measure and decay of correlation for covering weighted systems}, Ergodic
  Theory Dynam. Systems \textbf{18} (1998), no.~6, 1399--1420. \MR{1658635}

\bibitem[Mat95]{mattila}
Pertti Mattila, \emph{Geometry of sets and measures in {E}uclidean spaces},
  Cambridge Studies in Advanced Mathematics, vol.~44, Cambridge University
  Press, Cambridge, 1995, Fractals and rectifiability. \MR{1333890}

\bibitem[McM87]{mcmullen-area}
Curt McMullen, \emph{Area and {H}ausdorff dimension of {J}ulia sets of entire
  functions}, Trans. Amer. Math. Soc. \textbf{300} (1987), no.~1, 329--342.
  \MR{871679 (88a:30057)}

\bibitem[MU96]{mauldin-urbanski}
R.~Daniel Mauldin and Mariusz Urba{\'n}ski, \emph{Dimensions and measures in
  infinite iterated function systems}, Proc. London Math. Soc. (3) \textbf{73}
  (1996), no.~1, 105--154. \MR{1387085}

\bibitem[MU08]{mayer-urbanski-etds}
Volker Mayer and Mariusz Urba{\'n}ski, \emph{Geometric thermodynamic formalism
  and real analyticity for meromorphic functions of finite order}, Ergodic
  Theory Dynam. Systems \textbf{28} (2008), no.~3, 915--946. \MR{2422021
  (2010f:37080)}

\bibitem[MU10]{mayer-urbanski-mem}
\bysame, \emph{Thermodynamical formalism and multifractal analysis for
  meromorphic functions of finite order}, Mem. Amer. Math. Soc. \textbf{203}
  (2010), no.~954, vi+107. \MR{2590263}

\bibitem[Pat76]{patterson}
S.~J. Patterson, \emph{The limit set of a {F}uchsian group}, Acta Math.
  \textbf{136} (1976), no.~3-4, 241--273. \MR{0450547}

\bibitem[PRLS04]{PRS}
Feliks Przytycki, Juan Rivera-Letelier, and Stanislav Smirnov, \emph{Equality
  of pressures for rational functions}, Ergodic Theory Dynam. Systems
  \textbf{24} (2004), no.~3, 891--914. \MR{2062924}

\bibitem[Prz99]{conical}
Feliks Przytycki, \emph{Conical limit set and {P}oincar\'e exponent for
  iterations of rational functions}, Trans. Amer. Math. Soc. \textbf{351}
  (1999), no.~5, 2081--2099. \MR{1615954 (99h:58110)}

\bibitem[PU10]{pu-book}
Feliks Przytycki and Mariusz Urba{\'n}ski, \emph{Conformal fractals: ergodic
  theory methods}, London Mathematical Society Lecture Note Series, vol. 371,
  Cambridge University Press, Cambridge, 2010. \MR{2656475}

\bibitem[Rem09]{rempe-radial}
Lasse Rempe, \emph{Hyperbolic dimension and radial {J}ulia sets of
  transcendental functions}, Proc. Amer. Math. Soc. \textbf{137} (2009), no.~4,
  1411--1420. \MR{2465667}

\bibitem[Rue78]{ruelle-book}
David Ruelle, \emph{Thermodynamic formalism}, Encyclopedia of Mathematics and
  its Applications, vol.~5, Addison-Wesley Publishing Co., Reading, Mass.,
  1978, The mathematical structures of classical equilibrium statistical
  mechanics, With a foreword by Giovanni Gallavotti and Gian-Carlo Rota.
  \MR{511655}

\bibitem[Sta94]{stallard-dim}
Gwyneth~M. Stallard, \emph{The {H}ausdorff dimension of {J}ulia sets of
  meromorphic functions}, J. London Math. Soc. (2) \textbf{49} (1994), no.~2,
  281--295. \MR{1260113}

\bibitem[Sta99]{stallard}
\bysame, \emph{The {H}ausdorff dimension of {J}ulia sets of hyperbolic
  meromorphic functions}, Math. Proc. Cambridge Philos. Soc. \textbf{127}
  (1999), no.~2, 271--288. \MR{1705459}

\bibitem[Sul82]{sullivan}
Dennis Sullivan, \emph{Seminar on conformal and hyperbolic geometry}, Preprint
  IHES, 1982.

\bibitem[Tho12]{klaus}
Klaus Thomsen, \emph{K{MS} states and conformal measures}, Comm. Math. Phys.
  \textbf{316} (2012), no.~3, 615--640. \MR{2993927}

\bibitem[Urb03]{MU-measures}
Mariusz Urba{\'n}ski, \emph{Measures and dimensions in conformal dynamics},
  Bull. Amer. Math. Soc. (N.S.) \textbf{40} (2003), no.~3, 281--321.
  \MR{1978566}

\bibitem[UZ03]{urbanski-zdunik2}
Mariusz Urba{\'n}ski and Anna Zdunik, \emph{The finer geometry and dynamics of
  the hyperbolic exponential family}, Michigan Math. J. \textbf{51} (2003),
  no.~2, 227--250. \MR{1992945}

\bibitem[UZ04]{urbanski-zdunik1}
\bysame, \emph{Real analyticity of {H}ausdorff dimension of finer {J}ulia sets
  of exponential family}, Ergodic Theory Dynam. Systems \textbf{24} (2004),
  no.~1, 279--315. \MR{2041272}

\bibitem[UZ07]{urbanski-zdunik3}
\bysame, \emph{Geometry and ergodic theory of non-hyperbolic exponential maps},
  Trans. Amer. Math. Soc. \textbf{359} (2007), no.~8, 3973--3997. \MR{2302520}

\bibitem[VV10]{VV}
Paulo Varandas and Marcelo Viana, \emph{Existence, uniqueness and stability of
  equilibrium states for non-uniformly expanding maps}, Ann. Inst. H.
  Poincar\'e Anal. Non Lin\'eaire \textbf{27} (2010), no.~2, 555--593.
  \MR{2595192}

\end{thebibliography}

\end{document}